\documentclass[12pt]{amsart}
\usepackage[margin=1in]{geometry}
\usepackage{amsmath,amsfonts,amsthm,amssymb,bbm}
\usepackage{graphicx,color, dsfont}
\usepackage{fourier}

\newtheorem{theorem}{Theorem}

\newtheorem{lemma}{Lemma}
\newtheorem{proposition}{Proposition}

\newtheorem{corollary}{Corollary}

\newtheorem{claim}{Claim}

 \theoremstyle{definition}
 
 \theoremstyle{remark}

 \numberwithin{equation}{section}

\newcommand{\vertiii}[1]{{\left\vert\kern-0.25ex\left\vert\kern-0.25ex\left\vert #1
    \right\vert\kern-0.25ex\right\vert\kern-0.25ex\right\vert}}

\newcommand{\f}[2]{\frac{#1}{#2}}

\newcommand{\al}{\alpha}
\newcommand{\be}{\beta}

\newcommand{\ga}{\gamma}

\newcommand{\la}{\lambda}

\newcommand{\si}{\sigma}

\newcommand{\rone}{\mathbf R}

\newcommand{\dpr}[2]{\langle #1,#2 \rangle}

\newcommand{\eps}{\epsilon}

\newcommand{\p}{\partial}

\newcommand{\beq}{\begin{equation}}
\newcommand{\eeq}{\end{equation}}
\newcommand{\beqna}{\begin{eqnarray*}}
\newcommand{\eeqna}{\end{eqnarray*}}
\newcommand{\beqn}{\begin{equation*}}
\newcommand{\eeqn}{\end{equation*}}
\newcommand{\bp}{\begin{proof}}
\newcommand{\ep}{\end{proof}}
\newcommand{\bprop}{\begin{proposition}}
\newcommand{\eprop}{\end{proposition}}

\newcommand{\et}{\end{theorem}}
\newcommand{\bex}{\begin{Example}}
\newcommand{\eex}{\end{Example}}
\newcommand{\bc}{\begin{corollary}}
\newcommand{\ec}{\end{corollary}}
\newcommand{\bcl}{\begin{claim}}
\newcommand{\ecl}{\end{claim}}
\newcommand{\bl}{\begin{lemma}}
\newcommand{\el}{\end{lemma}}

\begin{document}

\title[Asymptotic stability for   traveling kinks of the   Boussinesq equation]
 {Asymptotic stability for monotone traveling kinks of the  dissipative  Boussinesq problem}
%----------Author 1

%\thanks{Stefanov is partially  supported by  NSF-DMS   under award \#  2204788}

\author[Atanas Stefanov]{\sc Atanas Stefanov}
\address{ }
\email{stefanov@uab.edu}

\subjclass[2000]{Primary  35Q35  Secondary 35B35, 35B40}

\keywords{traveling kinks}

\date{\today}

\begin{abstract}
  We study the dissipative Boussinesq problem, which extends the "good" Boussinesq equation by incorporating viscosity effects. It is well-known that this model supports monotone decreasing traveling kink solutions. We show that these kinks are orbitally stable in $H^1(\rone)$. Moreover, they are  asymptotically stable in $\dot{H}^1(\rone)$ and $L^p(\rone), 2<p\leq \infty$. 
\end{abstract}

\maketitle

\section{Introduction}
  We consider the Cauchy problem damped Boussinesq equation in one spatial dimension
  \begin{equation}
  	\label{10} 
  	\begin{cases}
  		u_{tt}- c u_{txx}+b u_{xxxx} - u_{xx}= \p_{xx} (f(u)), x\in \rone, t\geq 0 \\
  		u(0,x)=u_0(x);  u_t(0,x)=u_1(x).
  	\end{cases}
  	 \end{equation}
  where $b>0, c>0$. This belongs to a family of fundamental models for the propagation of long waves with small amplitude on the surface of shallow water, incorporating the viscosity effects of the medium. It is worth pointing out that the case $c=0$ reduces \eqref{10} to the ``good'' Boussinesq model, which was studied extensively in th last thirty years.  See \cite{TM} for the analysis of the Cauchy problem. However,  since the  early years of the water waves modeling, Whitham, \cite{Wh}  has wondered about  the precise form of the viscous correction effect. Hence, models like \eqref{10}, which feature a viscosity term analogous to the BBM equation, become meaningful objects of study. We refer to \cite{Liu1, Liu2, MPP, OC} for  results on local and global existence, long-time asymptotic behavior of solutions for the initial and  boundary value problems of Boussinesq equation.
  
  We now introduce some appropriate technical conditions on the non-linearity $f$ later on, but basically 
  we will be mostly concerned  with  the more concrete (and physical) case of integer power non-linearities\footnote{Although the nonlinearity assumption can be considerably more general} $f(u)=u^p, p=2, 3, \ldots$.  
We henceforth consider the physically relevant quadratic case 
    \begin{equation}
  	\label{11} 
  		u_{tt} - c u_{txx}+ b u_{xxxx} - u_{xx}=a \p_{xx} (u^2),  x\in \rone, t\geq 0.
  \end{equation}
Due to the availability of the transformation $u\to -u$, we may and do  assume that $a>0$. Using the transformation $u\to \f{1}{2a} u\left(\f{t}{c}, \f{x}{c}\right)$ puts the problem in the following convenient equivalent form 
 \begin{equation}
	\label{12} 
	u_{tt} - u_{txx}+ B u_{xxxx} - u_{xx}= \f{1}{2} \p_{xx} (u^2).
\end{equation}
where we have  denoted $B:=\f{b}{c^2}$. 
Our main goal is to understand the dynamics close to traveling kinks, that is - heteroclinic solutions of \eqref{12}. Note that the transformation $x\to -x$, which leaves \eqref{12} invariant allows us to assume that $u(-\infty)=\phi_->\phi_+=u(+\infty)$. 
\subsection{Traveling kinks for the viscous Boussinesq equation}

Assuming traveling wave ansatz with speed $\si$, namely 
$u(t,x)=\phi(x- \si t), \lim_{x\to \pm \infty}\phi(x)= \phi_\pm$, we obtain the ODE, 
\begin{equation}
	\label{14} 
	B \phi''''+\si  \phi'''+(\si^2-1) \phi''= \f{1}{2}  (\phi^2)'' 
\end{equation}
We focus on heteroclinic solutions   $\phi$, i.e.  kink solutions\footnote{Here, and thereafter, we do not distinguish between  fronts and backs, depending on $\phi_+>\phi_-$ or $\phi_+<\phi_-$. That is, we say that $\phi$ is a kink solution, if  $\lim_{x\to \pm \infty}\phi(x)=\phi_\pm, \phi_+\neq \phi_-$} with 
$$
\lim_{|x|\to \infty} |\phi'(x)|+|\phi'''(x)|=0, \lim_{x\to \pm \infty}\phi(x)=\phi_\pm
$$
 where generically $\phi_+\neq \phi_-$. 
 Integrating \eqref{14}  twice and introducing an integration constant k yields 
 \begin{equation}
 	\label{141} 
 	B \phi''+\si  \phi'+(\si^2-1) \phi= \f{1}{2}(\phi^2 + k).
 \end{equation}
 Taking limits at $\pm \infty$ shows  that $k$ must satisfy 
 \begin{equation}
 	\label{15} 
 	\f{k}{2} =(\si^2-1) \phi_+ - \f{1}{2} \phi_{+}^2=(\si^2-1) \phi_- -\f{1}{2} \phi_-^2
 \end{equation}
 Excluding $k$ from \eqref{15}, we arrive at the Rankine-Hugoniot relation between speed and limit values for the kink 
 \begin{equation}
 	\label{161} 
 	\si^2-1=\f{1}{2} \frac{\phi_+^2- \phi_-^2}{\phi_+-\phi_-}=\f{1}{2} (\phi_+ + \phi_-)
 \end{equation}
 This shows that a necessary condition relating the heteroclinic values $\phi_\pm$ for the problem \eqref{12},  is given by 
 \begin{equation}
 	\label{16} 
 	\phi_- + \phi_+\geq -2.
 \end{equation}
We are now ready to state our existence result. 
 \begin{proposition}[Existence for traveling kinks of the viscous Boussinesq]
 \label{prop:10} 
 Let $B>0$, $\phi_->\phi_+$ are three parameters, which satisfy \eqref{16}. Let $\si$ is given by \eqref{161}, 
 $$
 \si=\pm \sqrt{1+\f{1}{2} (\phi_+ + \phi_-)}.
 $$
  Then, the equation \eqref{14} has an  unique smooth solution $\phi$ with the property $\lim_{x\to \pm \infty} \phi(x) = \phi_\pm$, $\lim_{|x|\to \infty} |\phi'(x)|+|\phi''(x)| =0$. 
 Moreover, the fronts $\phi$ are monotonically decreasing \underline{if and only if} 
 \begin{equation}
 	\label{292} 
 	\nu:=\f{Bm}{\si^2}\leq \f{1}{4}, \ \ m:=\f{\phi_- - \phi_+}{2}
 \end{equation}
 Finally, the fronts obeys the following asymptotics at $\pm \infty$. For $\nu\in (0, 1/4)$, introduce the exponents $\mu_-(\nu)>0, \mu_+(\nu)<0$, 
	$$
	\mu_{\mp}(\nu)=\f{-1+\sqrt{1\pm 4\nu}}{2\nu}.
	$$
	Then, for some $c_\pm>0$, 
\begin{eqnarray}
	\label{f:1} 
	& & \begin{cases}
		\phi(x)=\phi_\pm - c_\pm exp(\f{m \mu_{\pm}(\nu)}{\si} x)+o(e^{\f{m \mu_{\pm}(\nu)}{\si} x}), \ \ x\sim \pm \infty \\
		\phi'(x)=- c_\pm \f{m \mu_{\pm}(\nu)}{\si} exp(\f{m \mu_{\pm}(\nu)}{\si} x)+o(e^{\f{m \mu_{\pm}(\nu)}{\si} x}), \ \ x\sim \pm \infty \\
		\phi''(x)=- c_\pm \left(\f{m \mu_{\pm}(\nu)}{\si}\right)^2 exp(\f{m \mu_{\pm}(\nu)}{\si} x)+o(e^{\f{m \mu_{\pm}(\nu)}{\si} x}), \ \ x\sim \pm \infty 
		\end{cases}. 
\end{eqnarray}

\end{proposition}
 We postpone the proof of Proposition \ref{prop:10} for the appendix, as it involves some technical details and transformations. 
 
 Next, we recast \eqref{12} in an equivalent form, which will be more useful in the sequel. 
 \subsection{Equivalent formulation of the viscous Boussinesq problem}  
We introduce a new unknown function $v$, 
$$
u_t-\al u_{xx}=:v_x,
$$
where $\al>0$ is to be determined shortly. Then, using the equation \eqref{12}, we obtain 
\begin{eqnarray*}
	\p_x (v_t - \be v_{xx}) &=& (\p_t-\be \p_{xx})(v_x) = (\p_t-\be \p_{xx})(u_t - \al u_{xx})=u_{tt}-(\al+\be) u_{txx} + \al \be u_{xxxx}=\\
	&=& (\al \be-B)u_{xxxx}+ u_{xx}+ \f{1}{2} \p_x^2(u^2).
\end{eqnarray*}
provided $\al, \be: c=\al+\be$, and we make the additional requirement that 
$\ga:=B-\al\be>0$. Note that such a choice of $\al, \be>0$  is always possible, as we need to solve 
$$
\begin{cases}
	\al+\be=1 \\
	\al \be = s<B
\end{cases}
$$
which has the pair of positive solutions 
$$
\al, \be=\f{1\pm \sqrt{1-4s}}{2},
$$
provided $s$ is selected, subject to  $0<s<\min(\f{1}{8}, \f{B}{2})$. The most suitable selection is to choose a small enough $0<\al<<1, \be=1-\al$, to ensure that the inequality $B>\al \be=\al(1-\al)$ holds. This is always possible, and any such choice is essentially equivalent; we therefore fix one such choice. 

With these choices, and after integrating the $w$ equation\footnote{Here, we need to ensure that $\lim_{|x|\to +\infty} (|u'(x)|+ |u'''(x)|)=0$}, we end up with the equivalent system 
\begin{equation}
	\label{28} 
	\begin{cases}
		u_t - \al u_{xx}= v_x \\
		v_t-\be  v_{xx}  = -\ga  u_{xxx}  + u_x + \f{1}{2} \p_x (u^2). 
	\end{cases}
\end{equation} 
Clearly, $v$ is determined up to a constant in \eqref{28}, as it always appears under a derivative. In fact, plugging in the special solution $u=\phi(x- \si t)$, we obtain, with $\psi'=-\si \phi'-\al \phi''$, or 
\begin{equation}
	\label{26}
	\psi=-\si  \phi-\al\phi', 
\end{equation}
where we have naturally enforced the requirement 
\begin{equation}
	\label{38} 
	\lim_{x\to \pm \infty} \psi(x)=-\si \phi_\pm.
\end{equation}

To recapitulate, suppose a 
 traveling kink $\phi$ (i.e. a solution of \eqref{14}) is given to us. 
So, in order to consider a well-defined problem for \eqref{28}, consistent with the relations just found, we need to analyze 
\begin{equation}
	\label{102} 
	\begin{cases}
		u_{tt}-  u_{txx}+ B u_{xxxx} - u_{xx}=\f{1}{2}  (u^2)_{xx}, x\in \rone, t\geq 0 \\
		u(0,x)=u_0(x);  u_t(0,x)=u_1(x)=:\p_x f_1;\\
		\lim\limits_{x\to \pm \infty} u_0(x)=\phi_\pm, \lim\limits_{x\to \pm \infty} f_1(x)=-\si \phi_\pm\\
		u_0-\phi, u_0' \in L^2(\rone), f_1 + \si  \phi \in L^2(\rone). 
	\end{cases}
\end{equation}
In such a case, we can formulate the corresponding problem for \eqref{28}, which is first order in time. Specifically,  
\begin{equation}
	\label{29} 
	\begin{cases}
		u_t - \al u_{xx}= v_x \\
		v_t-\be  v_{xx}  = -\ga  u_{xxx}  + u_x +\f{1}{2} (u^2)_x, \\
		u(0,x)=u_0(x); v(0,x)= f_1(x)-\al u_0'(x), \\
		u_0-\phi, u_0'\in L^2(\rone), f_1 + \si  \phi \in L^2(\rone).
	\end{cases}
\end{equation} 
We would like to explore the behavior of the model \eqref{12} close to the special solution $(\phi(x - \si t), \psi(x- \si t)$ of \eqref{29}.  Indeed, note that such a pair is a solution of \eqref{29}, with the boundary conditions $u_0=\phi, f_1=-\si  \phi$. 

Specifically, let
$$
\begin{cases}
	u(t,x)= \phi(x-\si t+\mu(t))+U(t, x-\si t +\nu(t)), \\
	v(t,x)= \psi(x-\si t+\mu(t)) +V(t, x-\si t + \nu(t))
\end{cases}
$$
where $\mu(t), \nu(t)$ are  $C^1$ functions, to be determined momentarily. We obtain the following nonlinear evolution problem, with localized Cauchy data\footnote{Note that in \eqref{105},  whenever one sees $\phi$, it is evaluated at the shifted argument $y+\mu(t)$, while $U,V$ are evaluated at the argument $y+\nu(t)$}
\begin{equation}
	\label{105} 
	\begin{cases}
		U_t +(\nu'(t)-\si)  U'+\mu' \phi' - \al U''= V' \\
		V_t+(\nu'(t)-\si) V' +\mu'\psi'-\be  V''  = -\ga  U''' + U' + \f{1}{2} (2\phi U+U^2)', \\
		U(0,y)=U_0(y) \in L^2(\rone) ; V(0,y)= V_0(y) \in L^2(\rone). 
	\end{cases}
	\end{equation} 
 Regarding our aim, we seek to show that upon a selection of appropriate $\mu=\mu(t)$ and  given small enough data $U_0, V_0$, the system \eqref{105} has global solutions $U, V$, so that they remain small for all forward times $t>0$ and in fact, converge to zero, in appropriate norms as $t\to +\infty$. 
\subsection{Main results}
Our main result establishes the asymptotic stability of the kinks $\phi(x-\si t)$,  stated precisely as follows. 
\begin{theorem}[Asymptotic stability of monotone fronts]
	\label{theo:10} 
	
	Let $B>0$, $\phi_- > \phi_+$ are so that \eqref{16} is satisfied. Let $\phi$ is the unique kink solution guaranteed by Proposition \ref{prop:10}. Assume in addition, 
	\begin{itemize}
		\item $\phi_->\phi_+>-1$
		\item The kink $\phi$ is monotonically decreasing, i.e. \eqref{292} holds true. 
	\end{itemize}
Then, there exists $\eps_0>0$, so that whenever $u_0, u_1=f_1'$ are functions satisfying 
\begin{equation}
	\label{d:10} 
	\|u_0-\phi\|_{H^1}+\|f_1+\si \phi\|_{L^2}<\eps_0,
\end{equation}
then, the initial value problem \eqref{102} has unique global smooth solution $u$, with the following properties: 
\begin{itemize}
	\item Orbital stability in $H^1$ norm, 
	$$
\sup_{t>0} 	\inf_{\mu\in \rone} \|u(t,x) - \phi(t, x-\si t +\mu)\|_{H^1}\leq 2 (\|u_0-\phi\|_{H^1}+\|f_1+\si \phi\|).
	$$
	\item Asymptotic stability in $\dot{H}^1$ norm, 
	$$
	\lim_{t\to +\infty}	\inf_{\mu\in \rone} \|u_x(t,x) - \phi'(t, x-\si t +\mu))\|_{L^2}=0.
	$$
	In particular, asymptotic stability holds in any $L^p, 2<p\leq \infty$ norm, 
	$$
	\lim_{t\to +\infty}	\inf_{\mu\in \rone} \|u(t,x) - \phi(t, x-\si t +\mu))\|_{L^p}=0.
	$$
\end{itemize}
\end{theorem}

\section{Preliminaries} 
This section establishes notation and preliminary results needed for the sequel.  First, the standard spatial Lebesgue spaces are denoted by 
$L^p(\rone)$, while the Sobolev space $H^1(\rone)$ also makes an appearance. Specifically, we record their respective norms\footnote{ Hereafter, we use the notation $\|f\|$ to denote $\|f\|_{L^2(\rone)}$.} 
$$
 \|f\|_p=\left(\int_{-\infty}^{+\infty} |f(x)|^p dx\right), \|f\|_{H^1}=\|f'\|+\|f\|,
$$
One can also define $H^k(\rone)=\{ f: \|f\|_{H^k}=\|f^{(k)}\|+\|f\| \}$ and $H^\infty(\rone)=\cap_{k=1}^\infty H^k(\rone)$. 
We define the Fourier transform and its inverse by
$$
\hat{f}(\xi)=\int_{-\infty}^{+\infty} f(x) e^{-2\pi i x \xi} dx, f(x) =\int_{-\infty}^{+\infty} \hat{f}(\xi) e^{2\pi i x \xi} d\xi. 
$$
We frequently employ the Sobolev embedding estimate and its refinement, the Gagliardo-Nirenberg inequality: 
\begin{eqnarray}
	\label{gn:10} 
& & 	\|u\|_{L^p}\leq C \|u\|_{H^1}, 2\leq p\leq \infty \\
	\label{gn:15} 
	& & \|u\|_{L^p}\leq C \|u'\|^{\f{1}{2}-\f{1}{p}} \|u\|^{\f{1}{2}+\f{1}{p}},  2\leq p\leq \infty.
\end{eqnarray}
In particular, $H^1(\rone)$ is an algebra,  since $\|u^2\|_{H^1(\rone)}\leq C \|u\|_{L^\infty} \|u\|_{H^1}\leq C  \|u\|_{H^1}^2$.  
\subsection{Some reductions for the viscous Boussinesq}
We now derive a solution formula for the forced viscous Boussinesq problem. Our goal is to establish a local well-posedness result for the nonlinear evolution problem \eqref{12}, subject to the boundary conditions as in \eqref{102}. To this end, consider kinks $\phi$, solving  \eqref{14}, and the decomposition $u(t,x)=\phi(x-\si t)+w(t,x)$. Then, $w$ solves 
\begin{equation}
	\label{400} 
	w_{tt} - w_{txx}+ B w_{xxxx} - w_{xx}= \f{1}{2} \p_{xx} (2\phi(\cdot-\si t)w+w^2).
\end{equation}
We supplement \eqref{400} with initial data $w_0(x)=w(0,x)\in H^1(\rone), w_1(x) = w_t(0,x)\in L^2(\rone)$. 

Based on this formulation, we consider the following forced linear viscous Boussinesq problem: 
\begin{equation}
	\label{410} 
\left\{\begin{matrix}
		q_{tt} - q_{txx}+ B q_{xxxx} - q_{xx}=\p_x^2(N(t,x)), \\
		q(0,x) \in H^1(\rone),  q_t(0,x)\in L^2(\rone). 
	\end{matrix}
	\right. 
\end{equation}
Note that the initial data for (2.4) are now localized, with the nonlinearity appearing under a second derivative—the natural formulation arising from the nonlinear problem
 \eqref{400}. The next step is to construct the Green's function. 
\subsection{Green's function for the free viscous Boussinesq problem}
Applying the spatial Fourier transform to \eqref{410}  with zero right-hand side yields the characteristic equation 
$$
\la^2 + 4\pi^2 \xi^2 \la +(16 \pi^4 \xi^4 B +4\pi^2 \xi^2)=0,
$$
which has the solutions 
$$
\la_\pm(\xi) =-2\pi^2\xi^2 \pm \sqrt{4\pi^4\xi^4(1-4 B)-4\pi^2\xi^2}.
$$
Note that $\la_\pm(\xi)$ is always complex valued for small $|\xi|$, however, under the assumption $B>0$, there is
\begin{equation}
	\label{411} 
	\Re\la_-<\Re\la_+<-c_0 \xi^2, 
\end{equation}
indicating a parabolic behavior. An elementary calculation shows that the free solution of \eqref{410} can be written in the form 

$$
	\hat{q}(t, \xi)=\f{\la_-(\xi) e^{t \la_+(\xi)}-\la_+(\xi) e^{t \la_-(\xi)}}{\la_-(\xi)-\la_+(\xi)} \hat{q}(0,\xi) + \f{e^{t \la_+(\xi)}-e^{t \la_-(\xi)}}{\la_+(\xi)-\la_-(\xi)} \hat{q}_t(0,\xi),
$$
where the second formula makes sense even for values $\xi: \la_+(\xi)=\la_-(\xi)$. 

A simple application of the Duhamel's formula yields the solution to the full problem \eqref{410} as follows 
\begin{eqnarray}
	\label{420} 
		\hat{q}(t, \xi) &=& \f{\la_-(\xi) e^{t \la_+(\xi)}-\la_+(\xi) e^{t \la_-(\xi)}}{\la_-(\xi)-\la_+(\xi)} \hat{q}(0,\xi) + \f{e^{t \la_+(\xi)}-e^{t \la_-(\xi)}}{\la_+(\xi)-\la_-(\xi)} \hat{q}_t(0,\xi) \\
		\nonumber
		& & -4\pi^2 \int_0^t \hat{G}(t-s, \xi) [\xi^2 \hat{N}(s, \xi)] ds,
\end{eqnarray}
where 
$$
\hat{G}(t,  \xi) =  \f{e^{t\la_+(\xi)}-e^{t\la_-(\xi)}}{\la_+(\xi)-\la_-(\xi)}. 
$$
We have the following Lemma.
\begin{lemma}
	\label{le:40} 
	There exists constants $C, c_0$, so that the  Green function $G$ satisfies 
	\begin{eqnarray}
		\label{450} 
	& & 	|\hat{G}(t, \xi)|\leq C(1+\xi^2)^{-1}, 	\\
		\label{451} 
	& & 	|\hat{G}(t, \xi)|\leq 
		C e^{-c_0 t |\xi|^2}
	\end{eqnarray}
\end{lemma}
\begin{proof}
	For large values of $|\xi|$, we have  $|4\pi^4\xi^4(1-4 B)-4\pi^2\xi^2|\sim \xi^4$, which means that $|\la_+(\xi)-\la_-(\xi)|\sim \xi^2$. So, for $|\xi|>>1$, 
	$$
	|\hat{G}(t, \xi)|=\f{|e^{t \la_-(\xi)}-e^{t  \la_+(\xi)}|}{|\la_+(\xi)-\la_-(\xi)|}\leq C\xi^{-2}.
	$$
where we have taken advantage of  the elementary inequality 
$|e^{t \la_\pm(\xi)}|=e^{t \Re \la_\pm(\xi)}\leq 1$. For values of $\xi$ ranging in a compact set $K$,  we have the inequality $|e^{t z}-1|\leq C_K|z|$ for all $z\in K$, whence 
	$$
|\hat{G}(t, \xi)|=e^{ t \Re \la_+(\xi)}\f{|e^{t  \la_-(\xi)-\la_+(\xi)}-1|}{|\la_+(\xi)-\la_-(\xi)|}\leq C.
$$
Either way, we obtain \eqref{450}. The bound \eqref{451}, which in view of \eqref{450}, is only relevant for large $|\xi|$, follows from \eqref{411}. 
\end{proof}
We are now ready for the local well-posedness result. 
\subsection{Local well-posedness for \eqref{400}}
\begin{proposition}
	\label{prop:54} 
	Suppose that $w_0\in H^1(\rone), w_1\in L^2(\rone)$. Then, there exists time \\  $T=T((\|w_0\|_{H^1(\rone)}, \|w_1\|_{L^2(\rone)})>0$ and a unique solution 
	$w\in C([0,T), H^1(\rone))$ of \eqref{400}. More specifically, $w$ satisfies the integral equation \eqref{420}, with the nonlinearity $N(t,x)=2\phi(x-\si t)w(t,x) +w^2(t,x)$. 
	A posteriori, the solution $w$ described above is in fact infinitely smooth, $w\in C^\infty([0,T), H^\infty(\rone))$. 
	
	The blow-up alternative for \eqref{400} is as follows - the solution $w$ blows up at time $T^*$ if and only if 
	\begin{equation}
		\label{90P}
		\limsup_{t\to T^*-} \|w(t, \cdot)\|_{H^1(\rone)}=+\infty. 
	\end{equation}
\end{proposition}
\begin{proof}
	The proof is based on a fixed point method for \eqref{420} for $w$, with the nonlinearity $N$, as specified. As is well-known, this reduces to estimates in the appropriate spaces. We have 
	\begin{eqnarray*}
	& & 	\|(1+|\xi|) \f{\la_-(\xi) e^{t \la_+(\xi)}-\la_+(\xi) e^{t \la_-(\xi)}}{\la_-(\xi)-\la_+(\xi)}  \hat{w}_0(\xi)\|_{L^2}\leq C \|(1+|\xi|) \hat{w}_0\|_{L^2}\leq C \|w_0\|_{H^1}, \\
	& & \|(1+|\xi|)  \f{e^{t \la_+(\xi)}-e^{t \la_-(\xi)}}{\la_+(\xi)-\la_-(\xi)} \hat{w}_1 (0,\xi)\|_{L^2}\leq C\|w_1\|_{L^2},
	\end{eqnarray*}
where in the second estimate, we have used \eqref{450}. 
For the Duhamel's term, we have by \eqref{450}, 
\begin{eqnarray*}
	& &  \|(1+|\xi|)  \int_0^t \hat{G}(t-s, \xi) [\xi^2 \hat{N}(s, \xi)] ds\|_{L^2}\leq C \int_0^t \| \xi^2  \hat{G}(t-s, \xi)  (1+|\xi|) \hat{N}(s, \xi)]\|_{L^2} ds \\
	&\leq & C \int_0^t \|(1+|\xi|) \hat{N}(s, \xi)\|_{L^2}\leq C t \sup_{0<s<t} [\|2\phi(\cdot-\si s) w\|_{H^1}+\|w^2(s, \cdot)\|_{H^1}]\leq \\
	&\leq & C t \sup_{0<s<t} 
	[\|w(s, \cdot)\|_{H^1}+ \|w(s, \cdot)\|_{H^1}^2].
\end{eqnarray*}
Selecting appropriately  small time $T\sim \f{1}{\|w_0\|_{H^1}+\|w_1\|_{L^2}+\|w_0\|_{H^1}^2+\|w_1\|_{L^2}^2}$ ensures a solution in the interval $[0,T]$, by the Banach fixed point theorem. The blow-up alternative \eqref{90P} follows from this argument as well, since one can clearly propagate the solution, by standard continuity  argument, past any time $T$, for which $	\limsup_{t\to T-} \|w(t, \cdot)\|_{H^1(\rone)}<+\infty$

Regarding the smoothness of the solution $w$, it suffices to observe that due to \eqref{411}, we have the estimate for the Green's function $|\hat{G}(t, \xi)|\leq C e^{-t \xi^2}$, similar to the heat kernel. It follows that the solution to the integral equation \eqref{420} experiences infinite smoothing effect, as long as $t>0$. Same argument applies to any number of time derivatives of $w$, whence $w\in C^\infty((0,T), H^\infty(\rone))$. {\it It is worth noting however that the infinte smoothness just established, applies only up to potential  blow-up time, and  in no way precludes the possibility of a finite time blow-up. } 
\end{proof}

\section{Asymptotic stability of the kinks}
Due to the local well-posedness result, Proposition \ref{prop:54}, we have a short time smooth solution of \eqref{102}. This translates in a short time  solution for \eqref{29} and so of \eqref{105}. Thus, we have a time, say $t_*\leq +\infty$, so that $U,V$ exist as smooth functions and  classical solutions of \eqref{105}. We now run a classical continuity argument to show that in fact $t_*<+\infty$ leads to a contradiction.  

Before we proceed with the proof, let us sort out the relation between initial data of the transformed problem \eqref{105}  $U_0, V_0$ and the initial data  of the original problem \eqref{102}. In fact, going through the concrete tranformations, we see that 
\begin{eqnarray*}
	U_0 (x) &=& u_0(x)-\phi(x);  \\
	V_0(x) &=& f_1(x) +\si \phi(x) - \al(u_0'(x)-\phi'(x)).
\end{eqnarray*}
Thus, the smallness assumptions \eqref{d:10} put forward in Theorem \ref{theo:10} translates into the following, more straightforward 
\begin{equation}
	\label{d:02} 
	\|U_0\|_{H^1}+\|V_0\|<<1. 
\end{equation}
This will be the basis of our contradiction argument below. 
\subsection{Energy estimates up to the eventual blow-up time}
As alluded above, we are running a contradiction argument, based on the assumption $t_*<+\infty$. To reiterate, the functions $U,V$ are classical smooth solutions of \eqref{105} for $t\in (0, t_*)$. Our next goal is to establish {\it a posteriori} estimates for $U,V$, which shows that blow-up cannot occur in any finite time. 

In order to focus on a better formulated result, we have the following Proposition. 
\begin{proposition}
	\label{prop:65} 
There exists $\eps_0=\eps_0(\phi, \al, \be, \ga)>0$, so that whenever 
\begin{equation}
	\label{297} 
	U_0\in H^1(\rone), V_0\in L^2(\rone):  \|U_0\|_{H^1}+\|V_0\|<\eps_0, 
\end{equation}
the solution $U,V$ to \eqref{105} extends globally in time. Moreover, they are smooth functions, and they satisfy the estimates 
\begin{eqnarray}
	\label{w:10} 
& & 	\|U(t)\|_{H^1}+ \|V(t)\|\leq C_1(\phi) (\|U_0\|_{H^1}+\|V_0\|) \\
		\label{w:20}
	& & 	\int_0^\infty  [\|U'(t)\|^2+ \|U''(t)\|^2+\|V'(t)\|^2]dt\leq C_2(\phi) (\|U_0\|_{H^1}^2+\|V_0\|^2). 
\end{eqnarray}
\end{proposition}
\begin{proof} 
	First, a word about the strategy of the proof. We run the usual continuity argument. Specifically,  we start the small data as in \eqref{297}, then analyze the local solutions, which exist thanks to  Proposition \ref{prop:54}. We need to show that \eqref{w:10} persists forever, where $C_1(\phi)$ will be determined explicitly in the course of the argument, see \eqref{419} below.  The trick is to use \\  $\sup_{0<t<t_*}	\|U(t)\|_{H^1}+ \|V(t)\|\leq C_1(\phi)  (\|U_0\|_{H^1}+\|V_0\|)$ as an {\it a priori} bound and show that until it is satisfied, which maybe up to some finite time  $t_*$, we have in fact a better bound, say $\limsup_{t\to t_*-} \|U(t)\|_{H^1}+ \|V(t)\|\leq \f{C_1(\phi)}{2} (\|U_0\|_{H^1}+\|V_0\|)$. This argument shows that the bound $\sup_{0<t<t_*}	\|U(t)\|_{H^1}+ \|V(t)\|\leq C_1(\phi)  (\|U_0\|_{H^1}+\|V_0\|)$ persists for all $t_*<\infty$. 
	
Let us now proceed to the actual estimates. Consider  \eqref{105} and  take dot product of the first equation with $U$ and the second equation with $V$, and add them up. Recall however that $\phi=\phi(y+\mu(t))$, whereas $U=U(y+\nu(t),t), V=V(y+\nu(t),t)$.
We obtain 
\begin{equation}
\label{110} 
\begin{cases}
	\f{1}{2} \p_t\left( \|U\|^2+\|V\|^2\right)+\mu'(\dpr{\phi'}{U}+\dpr{\psi'}{V})+ \al \|U'\|^2+\be\|V'\|^2= \\
	= -\ga\dpr{U'''}{V}+\f{1}{2} \dpr{\p_x (2\phi U+U^2)}{V}.
\end{cases}	
\end{equation}
We now compute various terms, by taking into account the relation \eqref{105}, 
\begin{eqnarray*}
	-\ga\dpr{U'''}{V} &=& \ga \dpr{U''}{V'}= \ga \dpr{U''}{U_t+(\nu'-\si)   U'+\mu'\phi'-\al U''}= \\
	&=& -\f{\ga}{2} \p_t \|U'\|^2+\ga\mu'\dpr{U}{\phi'''}-\ga \al \|U''\|^2. 
\end{eqnarray*}
On the other hand, integration by parts yields 
\begin{eqnarray*}
& & 	\f{1}{2}  \dpr{\p_x (2\phi U+U^2)}{V} = -\dpr{\phi U}{V'}-\f{1}{2}  \dpr{U^2}{V'}=  \\
	&=& - \dpr{\phi U}{U_t+(\nu'-\si) U'+\mu'\phi'-\al U''}- \f{1}{2}  \dpr{U^2}{U_t- (\nu'-\si)   U'+\mu'\phi'-\al U''}. 
\end{eqnarray*}
Term by term, and integration by parts 
\begin{eqnarray*}
	- \dpr{\phi U}{U_t} & = & -\f{1}{2} \p_t \left( \int \phi U^2 \right) +\f{1}{2}  \mu' \int \phi' U^2. \\
	(\si-\nu')  \dpr{\phi U}{U'} &=& \f{\nu'-\si}{2} \int \phi' U^2; \\
	- \mu' \dpr{\phi U}{\phi'}  &=&    -\mu' \int \phi\phi' U, \\
\al  \dpr{\phi U}{U''} &=& - \al \int \phi (U')^2 + \f{\al}{2}  \int \phi'' U^2. \\
-\f{1}{2} \dpr{U^2}{U_t} &=& -\f{1}{6} \p_t \int U^3(t,y) dy,\\
-\f{1}{2} \mu'\dpr{U^2}{\phi'} &=& -\f{1}{2} \mu'\int \phi' U^2,       \\ 
\f{\al}{2}   \dpr{U^2}{U''} &=& - \al \int  U (U')^2, 
\end{eqnarray*}
	All in all, we arrive at the identity, 
	\begin{eqnarray*}
	& &	\f{1}{2} \p_t\left( \int (1+ \phi) U^2+\f{1}{3} \int U^3+\|V\|^2+ \ga \|U'\|^2 \right)+\\
	& & + \al (\|U'\|^2+\int \phi (U')^2) + \be\|V'\|^2 + \ga \al \|U''\|^2+ \al \int  U (U')^2+ \\
& & + \mu'\left(\dpr{\phi'-\ga \phi'''+  \phi\phi'}{U} +\dpr{\psi'}{V}\right)+\f{1}{2}  \int  ((\si-\nu')  \phi' - \al \phi'')  U^2  =0.
	\end{eqnarray*}
 Here, the important terms to take care first are on the last line. More precisely, we need to make an appropriate selection of the functions $\mu, \nu$. First, we shall select $\nu(t):=N t+\si t$ for sufficiently large $N=N(\phi)$, to be determined in a moment. 
 
 The next step is to select $\mu'$, so that $\mu'\left(\dpr{\phi'-\ga \phi'''+  \phi\phi'}{U} +\dpr{\psi'}{V}\right)\ge 0$. Thus, the natural choice  is to take  
 \begin{equation}
 	\label{300}
 	\mu'(t)= \dpr{\phi'-\ga \phi'''+  \phi\phi'}{U} +\dpr{\psi'}{V}, \ \ \mu(0)=0. 
 \end{equation}
% \begin{equation}
% 	\label{300} 
% 	\begin{cases}
% 		
% 		=\int_{-\infty}^{+\infty} [(\phi'-\ga \phi'''+\phi\phi')(y+\mu(t)-\nu(t)) U(t,y) dy +    \psi'(y+\mu(t)-\nu(t)) V(t, y)] dy.
% 	\end{cases}
% \end{equation}
 Note that, 
 	\begin{eqnarray*}
 \dpr{\phi'-\ga \phi'''+  \phi\phi'}{U} +\dpr{\psi'}{V} &=&  \int_{-\infty}^{+\infty} (\phi'-\ga \phi'''+\phi\phi')(y+\mu(t)-\nu(t)) U(t,y) dy +  \\
 &+& 
\int_{-\infty}^{+\infty}  \psi'(y+\mu(t)-\nu(t)) V(t, y) dy=:G(t)
	\end{eqnarray*}
So, {\it assuming that we have control over $\|U(t, \cdot)\|, \|V(t, \cdot)\|$} till some time, say $t_*$, then the ODE for $\mu$ does not blow-up. In fact 
\begin{equation}
	\label{3001}
	\sup_{0<t<t_*}|\mu'(t)|=\sup_{0<t<t_*} |G(t)|\leq C \sup_{0<t<t_*}(\|U(t, \cdot)\|+ \|V(t, \cdot)\|),
\end{equation}
where $C=C(\phi)$. 

% 
%Indeed, \eqref{300} is an ODE for $\mu$, in the form $\mu'(t)=G(t, \mu(t))$, where $G$ is Lipschitz in the $\mu$ variable, due to the smoothness of $\phi$. Also, in order to justify the construction of $\mu(t)$, we shall need to establish that $G$ is continuous in the $t$ variable, which would follow from the Lipschitzness of $G$. To this end, it will suffice to establish that the following two functions 
%\begin{equation}
%	\label{310} 
%	t\to \int_{-\infty}^{+\infty} \psi'(y+\mu-\nu(t)) V_t(t, y) dy, t\to \int_{-\infty}^{+\infty} [(\phi'-\ga \phi'''+\phi\phi')(y+\mu-\nu(t)) U_t(t,y) dy
%\end{equation}
%are uniformly bounded up to a potential blow-up time $t_*$.   {\it That is, \eqref{310} still needs to be established in a pending continuity in time argument, see below}. 
With the choice of $\nu(t), \mu(t)$ as made in \eqref{300},  we arrive at the energy estimate 
\begin{equation}
	\label{320} 
	\begin{cases}
		\f{1}{2}  \p_t\left( \int (1+ \phi) U^2+\f{1}{3} \int U^3+\|V\|^2+ \ga \|U'\|^2 \right) + \al \int  U (U')^2\\
		+  \be\|V'\|^2 +  \ga \al \|U''\|^2+ \al  \int (1+ \phi) (U')^2 +\f{1}{2}  \int  (N (-\phi') - \al \phi'')  U^2  \leq 0.
	\end{cases}
\end{equation} 
Here, we have the easy Sobolev bound for the cubic term (for some absolute constant $C$), 
$$
 \left|\int  U (U')^2\right|  \leq \|U\|_{L^\infty} \|U'\|_{L^2}^2\leq C \|U(t)\|_{H^1} \|U'\|_{L^2}^2\leq C  C_1(\phi) \eps_0\|U'\|_{L^2}^2,
$$
due to the {\it a priori} bound on $\sup_{0<t<t_*} \|U\|_{H^1}\leq C_1(\phi)\eps_0$ and \eqref{297}. 
Additionally, we have the bound 
$$
 \f{\al}{2}   \int (1+ \phi) (U')^2 +\al  \int  U (U')^2\geq  \left(\f{\al(1+\phi_+)}{2} -CC_1(\phi) \eps_0\right)  \int (U')^2 >0, 
$$
for small enough $\eps_0$, namely $CC_1(\phi) \eps_0<\f{\al(1+\phi_+)}{4} $. 

We now tackle the remaining quadratic in $U$ terms, which appears in \eqref{320}. 
%$$
% \f{\al}{4}   \int (1+ \phi) (U')^2 +\f{1}{2}  \int  (N (-\phi') - \al \phi'')  U^2\geq  \f{1}{2}  \int  (N (-\phi') - \al \phi'')  U^2. 
%$$
 Recalling that the kink $\phi$ is decreasing in $(-\infty, +\infty)$, and hence $-\phi'>0$, we may select $N=N(\phi)$, so that $N (-\phi') - \al \phi''=(-\phi')\left(N-\al\f{\phi''}{-\phi'}\right)>0$ and hence $$
 \int  (N (-\phi') - \al \phi'')  U^2>0.
 $$
  Indeed, the asymptotics \eqref{f:1} show that 
$$
\lim_{x\to \pm \infty} \left|\frac{\phi''(x)}{\phi'(x)}\right|= \f{m|\mu_{\pm}(\nu)|}{\si}>0, 
$$
whence it follows,  that $\sup_{x} \left|\frac{\phi''(x)}{\phi'(x)}\right|<\infty$, as $\phi'$ does not vanish. So, a choice of $N:=2\al \sup_{x} \left|\frac{\phi''(x)}{\phi'(x)}\right|$ ensures that $N (-\phi') - \al \phi''>0$.

It follows that the last quadratic form is non-negative, which means that 
\begin{equation}
	\label{330} 
	\f{1}{2} \p_t\left( \int (1+ \phi) U^2+\f{1}{3} \int U^3+\|V\|^2+ \ga \|U'\|^2 \right) +  \be\|V'\|^2 +\f{\al(1+\phi_+)}{2} \|U'\|^2+  \ga \al \|U''\|^2\leq 0.
\end{equation}
This means that up to  $t_*$, we have control of the various norms, per \eqref{330} as follows. Ignoring the positive terms $ \be\|V'\|^2 +\f{\al(1+\phi_+)}{2} \|U'\|^2+  \ga \al \|U''\|^2$ for the moment, we  integrate\eqref{330}, to obtain for all $0<t<t_*$, 
\begin{eqnarray*}
& & 	 \int  (1+\phi(y+\mu(t)-\nu(t)) ) U^2(t, y)+\f{1}{3} \int U^3(t,y)+\|V(t,\cdot)\|^2+ \ga \|U'(t,\cdot)\|^2 \leq \\
	 &\leq &  \int  (1+\phi(y))  U_0^2(t, y)+\f{1}{3} \int U_0^3(y)+\|V_0\|^2+ \ga \|U'_0\|^2. 
\end{eqnarray*}
Trivially, for an absolute constant $C$, 
$$
|\int_{-\infty}^{+\infty}  U^3(t,y) dy|\leq C \|U\|_{L^\infty} \|U\|^2\leq C \|U(t)\|_{H^1}^3\leq C_1(\phi) \eps_0 \|U(t)\|_{H^1}^2, 
$$
due to the {\it a priori} bound for $\|U(t)\|_{H^1}$. 
For $\eps_0: C_1(\phi)\eps_0<\f{1}{100} \min(\ga, 1+\phi_+)$, we can hide the contribution of $\int U^3$. Indeed, 
\begin{eqnarray*}
	& & \f{\ga}{2} \|U'(t,\cdot)\|^2+  \f{1}{2} \int  (1+\phi(y+\mu(t)-\nu(t)) ) U^2(t, y) + \f{1}{3} \int U^3(t,y)\geq \\
	&\geq & \left(\f{\ga}{2}+\f{1+\phi_+}{2}\right) \|U\|_{H^1}^2 - C_1(\phi) \eps_0 \|U(t)\|_{H^1}^2\geq 0, 
\end{eqnarray*}
by the choice of $\eps_0$. So, we arrive at, 
\begin{equation}
	\label{414} 
	\f{\min((1+\phi_+), \ga)}{2} \|U(t, \cdot)\|_{H^1}^2 + \|V(t,\cdot)\|^2 \leq 	\max((1+\phi_-),\ga)  \|U_0\|_{H^1}^2+\|V_0\|^2. 
\end{equation}
It follows that 
\begin{equation}
	\label{418} 
	\sup_{0<t<t_*} (\|U(t, \cdot)\|_{H^1}+\|V(t,\cdot)\|)< \f{C_1(\phi)}{2}( \|U_0\|_{H^1}+\|V_0\|),
\end{equation}
where 
\begin{equation}
	\label{419} 
C_1(\phi)=100\sqrt{ \f{1+\max((1+\phi_-),\ga)}{\min(\f{\min((1+\phi_+), \ga)}{2},1)}}. 
\end{equation}
  Going back to \eqref{330}, we see that we get additional information. Namely, adding back on the terms $\be\|V'\|^2 +\f{\al(1+\phi_+)}{2} \|U'\|^2+  \ga \al \|U''\|^2$ and integrating (having in mind that the rest of the terms are positive and may be ignored), we get an $L^2_t$ estimate 
\begin{equation}
	\label{360} 
	\int_0^{\infty} [\|U'(t, \cdot)\|^2+\|V'(t, \cdot)\|^2 + \|U''(t,\cdot)\|^2] dt\leq C  [\|U_0\|_{H^1}^2 +
	\|V_0\|_{L^2}^2].
\end{equation}
This completes the proof of Proposition \ref{prop:65}. 
\end{proof}
\subsection{Asymptotic stability} 
We start this section with the observation that the statement of Proposition \ref{prop:65}, specificaly \eqref{w:10} is effectively a statement about orbital stability of the kinks $\phi$. We now extend this further in two ways - first we show that a stronger, asymptotic stability result holds\footnote{That is, appropriate norms of the residuals $U,V$ converge to zero as $t\to +\infty$} and also, we extend the range of the normed spaces in which this is true. It turns out that in this part of our argument, the {\it a priori} estimate \eqref{w:20} is more relevant. 
We claim that 
\begin{equation}
	\label{w:30} 
	\lim_{t\to +\infty} \|U'(t, \cdot)\|=0. 
\end{equation}
To this end,  note that the function $t\to \|U'(t, \cdot)\|^2$ is an integrable function, according to \eqref{w:20}. So, in order to establish \eqref{w:30}, it will suffice to see that the function $t\to \|U'(t, \cdot)\|^2$ is uniformly continuous on $(0, \infty)$. Since for $0<h<1$, 
	\begin{eqnarray*}
		& & \left| \|U'(t+h)\|^2-\|U'(t)\|^2\right|=2 |\int_{t}^{t+h} \p_\tau \dpr{U'(\tau)}{U'(\tau)} d\tau|\leq \\
		&\leq &  2 \left(\int_t^{t+1} \|\p_\tau U(\tau)\|^2 d\tau\right)^{1/2} 
		\left(\int_t^{t+1} \|U''(\tau)\|^2 d\tau\right)^{1/2}
\end{eqnarray*} and we already know from \eqref{w:20} that $\lim_{t\to \infty} \int_t^{t+1} \|U''(\tau)\|^2 d\tau=0$, it will suffice to show a uniform in $t$ bound on $\int_t^{t+1} \|\p_\tau U(\tau)\|^2 d\tau$. To this end we use the equation for $U_t$, 
%\begin{equation}
%	\label{w:40} 
%	\lim_{t\to\infty} \int_t^{t+1} |\p_s \|U'(s, \cdot)\|^2|  ds<\infty, 
%\end{equation}
%which is the focus henceforth. 
%We have by integration by parts and Cauchy-Schwartz, 
%\begin{eqnarray*}
%\f{1}{2} 	\int_t^{t+1} |\p_s \|U'(s, \cdot)\|^2|  ds= \int_t^{t+1} |\dpr{\p_s U(s)}{U''(s)}|  ds\leq  \left(\int_t^{t+1} \|U_s(s, \cdot)\|^2 ds\right)^{\f{1}{2}} \left(\int_t^{t+1} \|U''(s, \cdot)\|^2 ds\right)^{\f{1}{2}}.
%\end{eqnarray*}
%The estimate for $\int_t^{t+1} \|U''(s, \cdot)\|^2 ds<\int_0^\infty \|U''(s, \cdot)\|^2 ds<\infty$ is already available in \eqref{w:20}. For the other term, we make use of the equation for $U_s$ in the PDE, \eqref{105}. We have 
\begin{eqnarray*}
	\int_t^{t+1} \|U_s(s, \cdot)\|^2 ds &\leq &  C (N^2  \sup_{t<s<t+1} 
	\|U'(s, \cdot)\|_{L^2}+ \sup_{t<s<t+1} |\mu'(s)|^2 \|\phi'\|_{L^2}^2)+ \\
	&+& \al \int_0^\infty \|U''(s, \cdot)\|^2 ds+ \int_0^\infty \|V'(s, \cdot)\|^2 ds.
\end{eqnarray*}
 All of these terms are controlled in \eqref{w:10}, \eqref{3001} and subsequently by \eqref{w:20}. Thus, the mapping $t\to \|U'(t, \cdot)\|^2$ is uniformly continuous, and combined with $\int_0^\infty \|U'(t, \cdot)\|^2 dt<\infty$, implies that \\  $\lim_{t\to \infty} \|U'(t, \cdot)\|=0$. 

We now invoke the Gagliardo-Nirenberg estimate \eqref{gn:15}, to conclude that since 
$$
\|U(t)\|_{L^p}\leq C \|U'(t)\|^{\f{1}{2}-\f{1}{p}} \|U(t)\|^{\f{1}{2}+\f{1}{p}} 
$$
we have that for each $p: 2<p<\infty$
$$
\lim_{t\to \infty} \|U(t)\|_{L^p}=0.
$$

{\bf Funding declaration:} This work was supported by NSF under award \#  2204788.

\appendix

\section{Proof of Proposition \ref{prop:10}}
Starting with \eqref{141}, 
we apply  the transformation $\phi=\Phi+\si^2-1$.  This  leads to the form 
\begin{equation}
	\label{19} 
	B \Phi''+\si \Phi' - \f{1}{2}(\Phi^2-m^2)=0, 
\end{equation}
where $m=\f{\phi_- - \phi_+}{2}>0$. 
A  further reduction in \eqref{19}, namely, $\Phi=m q\left(\f{m}{\si}x\right)$,   transforms \eqref{19} into  
	\begin{equation}
		\label{291} 
		\f{m B}{\si^2} q''+q'-\f{1}{2}\left(q^2-1\right)=0.
	\end{equation}
	Note that \eqref{291} now asks for a function $q: \lim_{x\to \pm \infty} q(x) = \mp 1$. We now refer to the standard results on the topic,  \cite{BS85} and \cite{BRS}. Namely, each equation of the form \eqref{291} has an unique solution connecting $q(-\infty)=1, q(+\infty)=-1$. Moreover, such solutions are monotone decreasing if and only if $|\nu|=|\f{m B}{\si^2}|\leq \f{1}{4}$. 	This is exactly \eqref{292}. 
	
	Regarding the proof of the asymptotics \eqref{f:1} - these are standard results for heteroclinic solutions to second order equations on $\rone$. First, note that \eqref{f:1} follows from the transformation formula $\phi=mq\left(\f{m}{\si} x\right)+\si^2-1$ adopted above, provided one can show the asymptotics for $q$ as follows \begin{eqnarray}
			\label{f:2} 
			& & \begin{cases}
				q(x)=\mp 1 - c_\pm e^{\mu_{\pm}(\nu) x} +o(e^{\mu_{\pm}(\nu) x} ), \ \ x\sim \pm \infty \\
				q'(x)=  - c_\pm \mu_{\pm}(\nu) e^{\mu_{\pm}(\nu) x} +o(e^{\mu_{\pm}(\nu) x} ),\ \ x\sim \pm \infty \\
				q''(x)=- c_\pm \mu^2_{\pm}(\nu) e^{\mu_{\pm}(\nu) x}+o(e^{\mu_{\pm}(\nu) x} ), \ \ x\sim \pm \infty 
			\end{cases}.
	\end{eqnarray}
To this end, we study \begin{equation}\label{eq:ode}
	\nu q'' + q' - \frac{q^2-1}{2} = 0, \qquad q(-\infty)=1,\quad q(+\infty)=-1,
\end{equation}
Setting $v = u'$, equation~\eqref{eq:ode}
is equivalent to the planar autonomous system
\begin{equation}\label{eq:system}
	u' = v, \qquad
	\nu v' = \frac{u^2-1}{2} - v.
\end{equation}
Its equilibria are $(u,v)=(\pm 1, 0)$, corresponding to the two boundary states. 
At the left state, $u_0 = 1$, the characteristic equation becomes $\nu\lambda^2 + \lambda - 1 = 0$. The two roots
$$
\frac{-1 \pm \sqrt{1+4\nu}}{2\nu}
$$
are real  and they have opposite signs, 
for $\nu > 0$,  so $(1,0)$ is a saddle. 
The unique eigenvalue relevant to the kink,  is the positive root
$$
	  \mu_-(\nu)=\frac{-1 + \sqrt{1+4\nu}}{2\nu}>0.
$$
At the right state $u_0 = -1$ 
the characteristic equation  becomes $\nu\lambda^2 + \lambda + 1 = 0$, with 
the roots
$$
 \frac{-1 \pm \sqrt{1-4\nu}}{2\nu}
$$
which are real if and only if $\nu \le \tfrac{1}{4}$.
For  $\nu \in (0, 1/4)$,  they are both negative, so $(-1,0)$ is a stable node. 
The dominant eigenvalue,  governing the
 rate of approach,  is
$$
\mu_+(\nu) = \frac{-1 + \sqrt{1-4\nu}}{2\nu} < 0.
$$
These considerations imply the asymptotics \eqref{f:2}. Moreover, the requirement that the charactristic roots are all real is precisely what characterizes the monotonicity, so we see that $|\nu|<1/4$.
%\section{Proof of Lemma \ref{le:14}}
%Let $L_N:=-\p_x^2+N W_1 + W_2$ be the Schr\"odinger operator in question. Since $W_1, W_2$ are decaying potentials, it follows from the Weyl's theorem that 
%$$
%\si(L_N)=\si_{p.p.}(L_N)\cup \si_{a.c.}(L_N)=\si_{p.p.}(L_N)\cup [0, +\infty).
%$$
%Thus, we need to rule out negative eigenvalues for large enough values of $N$. To this end, assume for a contradiction that there are in fact a sequence of such eigenvalues $-\si_k<0$,  $N_k\to +\infty$, and with corresponding eigenfunctions $f_k\in H^2: \|f_k\|=1$. So, 
%$$
%-f_k''+N_k W_1 f_k+W_2 f_k=-\si_k f_k
%$$
%Taking dot product with $f_k$ yields 
%\begin{equation}
%	\label{c:10} 
%	\|f_k'\|^2+N_k \int W_1 f_k^2 + \si_k =-\int W_2 f_k^2\leq \|W_2\|_{L^\infty}
%\end{equation}
%As all the terms on the left-hand side are non-negative, it follows that $\{\si_k\}_k$ is a bounded sequence, while 
%\begin{equation}
%	\label{c:20} 
%	\sup_k [\|f_k\|_{H^1}+ \sqrt{N_k} \|\sqrt{W_1} f_k\|]<\infty.
%\end{equation}
%Recalling that $W_1$ is decaying at $\infty$, by Riesz-Kolmogorov criteria, the sequence $\{f_k\}$ is strongly compact in $L^2$. After taking appropriate subsequences (and assuming without loss of generality that $\lim_k \|f_k-f\|=0$, $f\in H^1: \|f\|=1$), we arrive at a contradiction with \eqref{c:20}, as $\lim_k \|\sqrt{W_1} f_k\|]=\|\sqrt{W_1} f\|]>0$. 
%
%
%
%
%
%
%
%
%
%
%
% 
%
%
%%\newpage 
% 


\begin{thebibliography}{99}
	
	
	
\bibitem{BRS} J. L. Bona, S. V. Rajopadhye, M. E. Schonbek, \emph{ Models for propagation of bores. I.Two-dimensional theory.}, {\em  Differential Integral Equations}, {\bf  7, (3-4)}, (1994),  p. 699--734. 

\bibitem{BS88}  J.L. Bona, R.L. Sachs, \emph{Global existence of smooth solutions and stability of solitary waves for a generalized Boussinesq equation}, {\em Comm. Math. Phys.} {\bf  118}, 
 (1988), p. 15--29.

\bibitem{BS85} J. L. Bona,  M. E. Schonbek, \emph{ Travelling-wave solutions to the Korteweg-de Vries-Burgers equation.}, {\em  Proc. Roy. Soc. Edinburgh Sect. A} {\bf 101},  (1985), no. 3-4,  p. 207--226. 
	
\bibitem{CD} W. 	Chen,  T. Dao, \emph{
	The Cauchy problem for the nonlinear viscous Boussinesq equation in the $l^q$ framework}, 
	{\em J. Differential Equations}, {\bf  320},  (2022), p. 558--597.
	
	\bibitem{HZL}   Y. Hu, W. Zhang, X. Ling, \emph{Qualitative analysis and bounded traveling wave solutions for Boussinesq equation with dissipative term}, {\em Nonlinear Dyn.}, {\bf 105}, (2021),  p. 2595--2609. 
	
	
	
\bibitem{LW} 	G. Liu, W. Wang, \emph{Decay estimates for a dissipative-dispersive linear semigroup and application to the viscous Boussinesq equation}, {\em J. Funct. Anal.} {\bf  278}, (7) (2020) 108413.

\bibitem{Liu1} Y. Liu, \emph{Instablity of solitary waves for generalized Boussinesq equations}, {\em J. Dynam. Differential Equations}, {\bf  5},  (1993), p. 537--558.
\bibitem{Liu2} Y. Liu, \emph{Instability and blow up of solutions to a generalized Boussinesq equation}, {\em SIAM J. Math. Anal.} {\bf  26},  (1995), p. 1527--1546.

\bibitem{MPP} C. Munoz, F. Poblete,  J. C. Pozo, \emph{ Scattering in the energy space for Boussinesq equations}, {\em Comm. Math. Phys.}, {\bf  361}, (2018), p. 127--141. 

\bibitem{OC} T. Ozawa,  Y. Cho, \emph{On small amplitude solutions to thegeneralized Boussinesq equations},  {\em Discrete Cont. Dyn. Sys. - A}, {\bf 17}, (2007), p. 691--711. 


\bibitem{TM} M. Tsutsumi,  T. Matahshi,  \emph{On the Cauchy problem for the Boussinesq-type equation}, {\em Math. Jpn. Soc.}, {\bf 36}, (1991), p. 371--379. 

\bibitem{Wh} G.B. Whitham, Linear and Nonlinear Wave. Wiley, NewYork (1974)



 
 
	
 

\end{thebibliography}
\end{document}